\documentclass{file}
\usepackage{pat}
\usepackage{paralist}
\usepackage{dsfont}  
\usepackage[utf8x]{inputenc}
\usepackage{skull}
\usepackage{centernot}
\usepackage{mathtools}
\usepackage{thmtools}
\usepackage{hyperref}
\usepackage[table]{xcolor}
\usepackage{stmaryrd,tikz}
\usetikzlibrary{calc}
\usepackage{graphicx}
\usepackage{booktabs}
\usepackage{listings}
\usepackage{pifont}%
\usepackage[normalem]{ulem}
\usepackage{cancel,float}
\usepackage{comment}

\usepackage{todonotes}
\usepackage{mathdots}
\usepackage{bm}
\usepackage{etoolbox}
\definecolor{maroon}{rgb}{0.5, 0.0, 0.0}
\definecolor{darkblue}{rgb}{0.0, 0.0, 0.55}

\usepackage[longnamesfirst,numbers,sort&compress]{natbib}

\usepackage[mathlines]{lineno}
\newcommand*\patchAmsMathEnvironmentForLineno[1]{%
 \expandafter\let\csname old#1\expandafter\endcsname\csname #1\endcsname
 \expandafter\let\csname oldend#1\expandafter\endcsname\csname end#1\endcsname
 \renewenvironment{#1}%
    {\linenomath\csname old#1\endcsname}%
    {\csname oldend#1\endcsname\endlinenomath}}%
\newcommand*\patchBothAmsMathEnvironmentsForLineno[1]{%
 \patchAmsMathEnvironmentForLineno{#1}%
 \patchAmsMathEnvironmentForLineno{#1*}}%
\AtBeginDocument{%
\patchBothAmsMathEnvironmentsForLineno{equation}%
\patchBothAmsMathEnvironmentsForLineno{align}%
\patchBothAmsMathEnvironmentsForLineno{flalign}%
\patchBothAmsMathEnvironmentsForLineno{alignat}%
\patchBothAmsMathEnvironmentsForLineno{gather}%
\patchBothAmsMathEnvironmentsForLineno{multline}%
}

\definecolor{brightmaroon}{rgb}{0.76, 0.13, 0.28}
\definecolor{linkblue}{rgb}{0, 0.337, 0.227}
\newcommand{\cay}{\mathop{\mathsf{Cay}}}

\newcommand{\Aut}{\mathsf{Aut}}

\newcommand{\edge}{\mathord{\,\vrule width 12pt height 3pt depth -1.5pt}\,}
\newrobustcmd{\onesub}{\mathord{\includegraphics{one-sub}}}
\newrobustcmd{\leftup}{\mathord{\includegraphics{left-up}}}
\makeatletter
\newcommand{\xMapsto}[2][]{\ext@arrow 0599{\Mapstofill@}{#1}{#2}}
\def\Mapstofill@{\arrowfill@{\Mapstochar\Relbar}\Relbar\Rightarrow}
\makeatother

\title{\centering\MakeUppercase{On Matrix Product Factorization of Cayley graphs}}
\author{
\begin{tabular}{c}
\textbf{Allen W. Herman} \\[3pt]
Department of Mathematics and Statistics \\ 
University of Regina
\end{tabular}
\quad\quad
\begin{tabular}{c}
\textbf{Bobby Miraftab} \\[3pt]
School of Computer Science \\ 
Carleton University
\end{tabular}
}

\date{}

\begin{document}

\maketitle

\begin{abstract}
We study when the adjacency matrix of a Cayley graph factors as the product of two adjacency matrices of Cayley graphs. Let $G$ be a finite group and let $U\subseteq G\sm \{e\}$ be symmetric. 
Writing $A(G;U)$ for the adjacency matrix of the Cayley graph of $G$ with respect to $U$, we prove that for symmetric subsets $S,T,U$ of  $G\sm \{e\}$, $A(G;U)=A(G;S)\,A(G;T)$
if and only if $U=ST$ and each $u\in U$ has a unique representation $u=st$,
equivalently $\bigl(\sum_{s\in S}s\bigr)\bigl(\sum_{t\in T}t\bigr)=\sum_{u\in U}u$ in the group algebra. When $S,T,U$ are unions of conjugacy classes, this is characterized character-theoretically by $\chi(U)=\chi(S)\chi(T)/\chi(1)$ for all $\chi\in\mathrm{Irr}(G)$. 
In addition, for abelian groups, we identify $A(G;S)A(G;T)$ with the $0\!-\!1$ convolution $\mathbf{1}_S*\mathbf{1}_T$, so factorability is equivalent to $(S,T)$ being a Sidon pair, i.e., $(S-S)\cap(T-T)=\{0\}$. 
For cyclic groups, we reformulate factorability via mask polynomials and reduce to prime-power components using the Chinese Remainder Theorem. 
We also analyze dihedral groups $D_{2n}$, presenting infinite families of factorable generating sets, and give explicit constructions of subsets whose Cayley graphs do and do not admit such factorizations.
\end{abstract}

\section{Introduction}

A classical topic in graph theory is to ask when a graph can be written as a product of smaller graphs with respect to a fixed graph product (Cartesian, tensor, strong, etc.).  
For instance, a graph $X$ has a \emph{Cartesian factorization} if $X \cong H \,\Box\, K$ for some graphs $H,K$, see \citet{MR2817074} for background and many results across product types.
In a different direction, \citet{Manjunatha} introduced the \defin{matrix product} of graphs: if $H$ and $K$ are graphs on the common vertex set $[n]=\{1,\dots,n\}$ with adjacency matrices $B$ and $C$, their matrix product $HK$ is the (generally weighted, directed) graph whose adjacency matrix is the usual matrix product $BC$.  
We refer readers to the subsequent works for variants, constraints, and applications \cite{maghsoudi2023matrix,akbari2025spectral,miraftab2025factorability}.
Throughout this paper we are interested in simple, undirected graphs, so we focus on situations where the matrix product again yields a simple $0\!-\!1$ adjacency matrix.

Let $G$ be a finite group with identity $e$. For a subset $U\subseteq G$ with $e\notin U$ and $U=U^{-1}$, the \defin{Cayley graph} $\cay(G;U)$ has vertex set $G$ and edges $g\edge h$ if and only if $g^{-1}h\in U$. 
A subset $ U $ of $ G $ is called \defin{symmetric} if $ U^{-1} = U $.
Fix once and for all an enumeration $G=\{g_1,\dots,g_n\}$ and indexed matrices by this order; the adjacency matrix is
\[
A(G;U)=(a_{ij})_{1\le i,j\le n},\qquad a_{ij}=\mathbf 1_{\,g_i^{-1}g_j\in U}.
\]
A central question of this paper is:

\begin{quote}
\emph{When does a Cayley adjacency $A(G;U)$ factor as a matrix product of Cayley adjacencies,}
\[
A(G;U)=A(G;S)\,A(G;T)\qquad\text{with }S,T\subseteq G\sm \{e\}\text{ symmetric?}
\]
\end{quote}

We emphasize that we are insisting that our factors and their product are Cayley graphs on the same group $G$ under a fixed Cayley representation.  Indeed, a product of two arbitrary simple graphs on $G$ may well be a Cayley graph on $G$ even when neither factor is.  The following example over $G=\Z_4=\{0,1,2,3\}$ illustrates this.
Let $H$ have edges $\{0\!\edge\!1,\ 2\!\edge\!3\}$ and let $K$ have edges $\{1\!\edge\!2,\ 3\!\edge\!0\}$.  With respect to the order $0,1,2,3$,
\[
A(H)=\begin{pmatrix}
0&1&0&0\\
1&0&0&0\\
0&0&0&1\\
0&0&1&0
\end{pmatrix},\qquad
A(K)=\begin{pmatrix}
0&0&0&1\\
0&0&1&0\\
0&1&0&0\\
1&0&0&0
\end{pmatrix}.
\]
A direct multiplication gives
\[
A(H)A(K)=\begin{pmatrix}
0&0&1&0\\
0&0&0&1\\
1&0&0&0\\
0&1&0&0
\end{pmatrix},
\]
which is the adjacency of the matching $\{0\!\edge\!2,\ 1\!\edge\!3\}=\cay(\Z_4;\{2\})$.  However, neither $H$ nor $K$ is a Cayley graph on $\Z_4$ (and no relabeling by automorphisms/right translations of $\Z_4$ turns them into Cayley graphs on $\Z_4$).  On the other hand, all three of $A(H)$, $A(K)$, and $A(H)A(K)$ above can be regarded as Cayley graphs on $\Z_2 \times \Z_2$, and relative to this group they do give an example of a Cayley-by-Cayley factorization. 
Here we introduce the following:
\begin{defn}
Let $G $ be a finite group and let $S,T,U\subseteq G\sm \{e\}$ be symmetric subsets such that $A(G;U)=A(G;S)A(G;T)$. 
Then we say that $\cay(G;U)$ has a \defin{factorization}.
In this case, we say that $(S,T,U)$ is \defin{factorable} in $G$.
\end{defn}

Following \citet{decaen}, a \defin{near–factorization} of a finite group $G$
is a pair $(S,T)$ with $e\notin S\cup T$ such that every $g\in G\sm \{e\}$ has a
\emph{unique} representation $g=st$ with $s\in S$, $t\in T$, equivalently
$ST=G\sm \{e\}$ and $S\cap T=\emptyset$.
In our language, this is the special case $U=G\sm \{e\}$, for which
$\bigl(\sum_{s\in S}R(s)\bigr)\bigl(\sum_{t\in T}R(t)\bigr)=\sum_{g\in G}R(g)-I$.  
This provides a useful benchmark
for our sparser Cayley‑by‑Cayley factorizations.

Restricting to Cayley factors leads to a rigid and tractable algebraic formulation.  Writing $R(g)$ for the right‑regular permutation matrix, we have
\[
A(G;X)=\sum_{x\in X} R(x)\qquad(X\subseteq G),
\]
so
\[
A(G;S)A(G;T)=\sum_{(s,t)\in S\times T} R(st).
\]
Since $\{R(g):g\in G\}$ are linearly independent, the identity
$
A(G;U)=A(G;S)A(G;T)
$
holds \emph{iff} every element of $U$ is realized \emph{uniquely} as a product $st$ with $s\in S$, $t\in T$, and no element outside $U$ is so realized.  Equivalently, in the group algebra,
\[
\Big(\sum_{s\in S}s\Big)\Big(\sum_{t\in T}t\Big)=\sum_{u\in U}u.
\]

We note that near–factorizations of groups correspond to factorizations of $J-I$, but the group structure
imposes additional constraints absent from arbitrary $(0,1)$-matrix factorizations, see
\citet{decaen}.


\section{Cayley Factorization of General Finite Groups}

Throughout this paper, $G$ denotes a finite group, and we always assume that the subsets $S$, $T$, and $U$ are symmetric subsets of $G$.

\begin{lem}\label{lem:relabel}
Let $G,G'$ be finite groups and let $\varphi\colon G\to G'$ be an isomorphism.
Let $P_\varphi$ be the $n\times n$ permutation matrix with
$(P_\varphi)_{g,\varphi(g)}=1$ for all $g\in G$.
Then for every subset $X\subseteq G'$,
\[
P_\varphi^{\!\top}\,A(G';X)\,P_\varphi \;=\; A\bigl(G;\varphi^{-1}(X)\bigr).
\]
\end{lem}

\begin{proof}
Write $A'=A(G';X)$ and $A=A\bigl(G;\varphi^{-1}(X)\bigr)$.
For $g,h\in G$ we have
$\bigl(P_\varphi^{\!\top}A'P_\varphi\bigr)_{g,h}=A'_{\varphi(g),\varphi(h)}$.
By the definition of Cayley adjacency, we have the following:
\[
A'_{\varphi(g),\varphi(h)}=1 \iff \varphi(g)^{-1}\varphi(h)\in X
\iff \varphi(g^{-1}h)\in X
\iff g^{-1}h\in \varphi^{-1}(X)
\iff A_{g,h}=1.
\]
Hence $P_\varphi^{\!\top}A(G';X)P_\varphi = A(G;\varphi^{-1}(X))$.
\end{proof}

If $\varphi$ is merely a bijection (not a homomorphism), the identity can fail.
For example, take $G=G'=\Z_3=\{0,1,2\}$ (additive), let $\varphi$ swap $0$ and $1$ and fix $2$,
and let $X=\{2\}$. Then for $g=1,h=0$,
\[
\bigl(P_\varphi^{\!\top}A(G';X)P_\varphi\bigr)_{1,0}
=0,
\]
while
\[
A\bigl(G;\varphi^{-1}(X)\bigr)_{1,0}=1,
\]
so the matrices are different.

The preceding lemma shows that, up to isomorphism, it suffices to study factorization within the class of Cayley graphs.
Throughout this paper, we write $\mathbf 1_X$ for the indicator of a set $X$.

\begin{defn}
Let $G$ be a finite group and $A$ an associative (not necessarily commutative) algebra over a field.
A map $R\colon G\to A$ is called a \defin{(linear) representation} (or \defin{algebra homomorphism}) if
\[
R(xy)=R(x)\,R(y)\quad\text{for all }x,y\in G,\qquad\text{and}\qquad R(e)=\mathbf{1}_A,
\]
where $e$ is the identity of $G$ and $\mathbf{1}_A$ is the multiplicative identity of $A$.
\end{defn}

Let $A$  be the algebra of all matrices index rows/columns by $G$ over $\mathbb C$ i.e.  $M_{|G|}(\mathbb{C})$.
Define $R\colon G\to M_{|G|}(\mathbb{C})$ by
\[
R(x)_{u,v}=
\begin{cases}
1,&\text{if } v=ux,\\
0,&\text{otherwise},
\end{cases}
\qquad (u,v\in G).
\]
Then $R(xy)=R(x)R(y)$ for all $x,y\in G$; hence $R$ is a representation.
We note that 

\begin{lem}\label{lem:product}
Let $G$ be a group, let $S,T\subseteq G\sm \{e\}$ be finite sets, and let $R\colon G\to A$ be a representation into an associative algebra $A$.
Then
\[
\sum_{s\in S}\ \sum_{t\in T} R(s)\,R(t)
\;=\;
\sum_{s\in S}\ \sum_{t\in T} R(st).
\]
\end{lem}

\begin{proof}
Since $R$ is a homomorphism, for each $s\in S$ and $t\in T$ we have $R(s)R(t)=R(st)$.
Therefore the two double sums agree term by term over the common index set $S\times T$:
\[
\sum_{(s,t)\in S\times T} R(s)R(t)=\sum_{(s,t)\in S\times T} R(st),
\]
which proves the claim.
\end{proof}

\begin{lem}\label{lem:product_2}
Let $G$ be a finite group and let $S,T,U \subseteq G\sm \{e\}$. 
If $U=ST$ and moreover for every $u\in U $, there exists a unique pair $(s,t)\in S\times T$ such that $u=st$, then $|U|=|S||T|$.
In addition if $S$ and $T$ are symmetric, then $S\cap T$ is an empty set.
\end{lem}

\begin{proof}
Define $\mu:S\times T\to U$ by $\mu(s,t)=st$. Since $U=ST$, $\mu$ is surjective.
The uniqueness hypothesis says each $u\in U$ has exactly one preimage, so $\mu$ is injective.
Thus $\mu$ is a bijection and $|U|=|S||T|$.
Now assume $S$ and $T$ are symmetric. If $x\in S\cap T$, then $x\in S$ and $x^{-1}\in T$,
so $e=xx^{-1}\in ST=U$, contradicting $U\subseteq G\sm\{e\}$. Hence $S\cap T=\varnothing$.
\end{proof}

\begin{thm}\label{thm:main_1}
Let $G$ be a finite group and let $S,T,U \subseteq G\sm \{e\}$ be symmetric sets. 
Then the following statements are equivalent:
\begin{enumerate}
    \item $(S,T,U)$ is factorable in $G$.
    \item $U=ST$ and moreover for every $u\in U $, there exists a unique pair $(s,t)\in S\times T$ such that $u=st$.
    \item $(\sum_{s\in S} s)(\sum_{t\in T}t )=\sum_{u\in U}u$.
\end{enumerate}
\end{thm}

\begin{proof}
Let $R$ be the right-regular representation. For any $X\subseteq G$,
$A(G;X)=\sum_{x\in X}R(x)$. Hence
\[
A(G;S)A(G;T)=\sum_{s\in S}\sum_{t\in T}R(st)=\sum_{g\in G}N(g)\,R(g),
\]
where $N(g)=|\{(s,t)\in S\times T:st=g\}|$.
Since $\{R(g):g\in G\}$ is linearly independent,
\[
A(G;U)=A(G;S)A(G;T)\iff
N(g)=\begin{cases}1,&g\in U,\\0,&g\notin U.\end{cases}
\]
This is exactly the statement that $U=ST$ and each $u\in U$ has a unique representation $u=st$.
This proves (1)$\Leftrightarrow$(2).
Moreover,
\[
\Big(\sum_{s\in S}s\Big)\Big(\sum_{t\in T}t\Big)=\sum_{g\in G}N(g)\,g
\]
in $\Z[G]$, so the same coefficient condition is equivalent to
$\bigl(\sum_{s\in S}s\bigr)\bigl(\sum_{t\in T}t\bigr)=\sum_{u\in U}u$.
Thus (2)$\Leftrightarrow$(3), completing the proof.
\end{proof}

\begin{cor}
For $U=G\sm\{e\}$, the identity $A(G;U)=A(G;S)A(G;T)$ holds if and only if $(S,T)$ is a symmetric near–factorization of $G$ in the sense of \cite{decaen}.
In particular $|S||T|=|G|-1$ and
$S\cap T=\varnothing$.
\end{cor}

We emphasize that by factorization in this paper, we mean each factor is a Cayley graph.


\begin{exa}
Let $G=N\rtimes_{\theta}\langle x\rangle$ with $x^2=e$ acting by an involutive automorphism $\theta(n)=xnx^{-1}$ on $N$. Let $U\subseteq N\sm\{e\}$ be symmetric and  $xUx^{-1}=U$. 
We set $S=\{x\}$ and $T=Ux$. 
Then $(S,T,U)$ is factorable in $G$.
Moreover, if $\langle U,x\rangle=G$ then $\cay(G;U)$ is connected and factorable.
\end{exa}

\begin{lem}\label{prop:stability}
Let $G$ be a finite group and let $U,S,T\subseteq G$.
If $(S,T,U)$ is factorable in $G$, then for every $g\in G$ we also have,
$A(G;gUg^{-1}) \;=\; A(G;gSg^{-1})\,A(G;gTg^{-1})$.
Moreover, for every automorphism $\varphi\in \Aut(G)$,
$A(G;\varphi(U)) \;=\; A(G;\varphi(S))\,A(G;\varphi(T))$.
\end{lem}

\begin{proof}
For $x\in G$, let $P_x$ be the permutation matrix of the bijection
$G\to G$, $v\mapsto vx$. Then one checks directly (by entrywise
evaluation) that $P_g^{-1}\,A(G;X)\,P_g \;=\; A(G;g^{-1}Xg)$, $\forall\,X\subseteq G$.
Conjugating the identity $A(G;U)=A(G;S)A(G;T)$ by $P_g$ gives
$$A(G;gUg^{-1})=A(G;gSg^{-1})A(G;gTg^{-1})$$
Similarly, an automorphism $\varphi$ induces a permutation matrix $P_\varphi$
(relabeling vertices by $v\mapsto \varphi(v)$), and again
$P_\varphi^{-1}\,A(G;X)\,P_\varphi \;=\; A(G;\varphi(X))$.
Conjugating by $P_\varphi$ yields
$A(G;\varphi(U))=A(G;\varphi(S))A(G;\varphi(T))$.
\end{proof}

\begin{cor}
Assume that $G$ is a finite group.
If $(S,T,U)$ is factorable in $G$ and $g \in G$, then $(gSg^{-1}, gTg^{-1}, gUg^{-1})$ is factorable in $G$.
\end{cor}

We adopt the following terminology from \cite[§1]{KreherPatersonStinson2025}.
Although the terminology was originally introduced for near-factorizations, it can be naturally extended to our setting.

\begin{defn}
Let $G$ be a group and let $(S,T,U)$ and $(S',T',U')$ be factorable in $G$.
We say they are \defin{equivalent} if there exists $\varphi\in\Aut(G)$ such that
$(S',T',U')=(\varphi(S),\varphi(T),\varphi(U))$.
\end{defn}



\begin{thm}\label{thm:normal}
Let $G$ be a finite group. 
Suppose symmetric subsets $S,T,U\subseteq G\sm\{e\}$ are unions of conjugacy classes. Then $G$ has the factorization
$A(G;U)=A(G;S)A(G;T)$ if and only if $\chi(U)=\chi(S)\,\chi(T)/\chi(1)$ for every $\chi\in\mathrm{Irr}(G)$,
where $\chi(X)\coloneqq\sum_{x\in X}\chi(x)$. In particular for every non-trivial $\chi$ with $\chi(1)=1$,
\[
|\chi(U)|=|\chi(S)|\,|\chi(T)|\le |S|\,|T|=|U|.
\]
\end{thm}

\begin{proof}
For $X\subseteq G$, set $\sigma_X\coloneqq  \sum_{x\in X}x\in\mathbb{C}[G]$.
If $X$ is a union of conjugacy classes then $\sigma_X$ is central in $\mathbb{C}[G]$.
Let $\rho$ be an irreducible complex representation with character $\chi$ and degree $\chi(1)$.
Since $\sigma_X$ is central, Schur's Lemma implies $\rho(\sigma_X)=\lambda_X(\chi)\,I_{\chi(1)}$
for some scalar $\lambda_X(\chi)$.
Taking traces gives
\[
\chi(X)=\sum_{x\in X}\chi(x)=\mathrm{tr}\bigl(\rho(\sigma_X)\bigr)
=\mathrm{tr}\bigl(\lambda_X(\chi)I_{\chi(1)}\bigr)=\lambda_X(\chi)\chi(1),
\]
so $\lambda_X(\chi)=\chi(X)/\chi(1)$.
Now $A(G;X)$ corresponds to $\rho(\sigma_X)$ under the regular representation embedding,
so for every $\chi\in\mathrm{Irr}(G)$ we have the eigenvalue identity
\[
\rho(\sigma_U)=\rho(\sigma_S)\rho(\sigma_T)
\iff
\lambda_U(\chi)=\lambda_S(\chi)\lambda_T(\chi).
\]
Substituting $\lambda_X(\chi)=\chi(X)/\chi(1)$ yields
\[
\frac{\chi(U)}{\chi(1)}=\frac{\chi(S)}{\chi(1)}\cdot\frac{\chi(T)}{\chi(1)}
\iff
\chi(U)=\frac{\chi(S)\chi(T)}{\chi(1)}.
\]
This holds for all $\chi\in\mathrm{Irr}(G)$ if and only if $A(G;U)=A(G;S)A(G;T)$.
Finally, for $\chi(1)=1$ we have $|\chi(X)|\le |X|$, hence
$|\chi(U)|=|\chi(S)\chi(T)|\le |S||T|=|U|$.
\end{proof}

\begin{lem}\label{lem:xUx=U}
Let $U\subseteq G\sm\{e\}$ be a symmetric subset of a group $G$. 
If $G$ has an involution $x\notin U$ such that $xUx=U$, then $A(G;U)$ has a factorization. 
\end{lem}

\begin{proof}
We apply \Cref{thm:main_1} with $S\coloneqq\{x\}$ and $T\coloneqq xU$.  Since $x$ has order $2$ and $xUx=U$, we have that for all $u \in U$, $(xu)^{-1} = u^{-1}x = x(xu^{-1}x) = x(xux)^{-1} \in xU$.  So both $S=\{x\}$ and $T=xU$ are symmetric subsets of $G$.  
One can see that 
$ST=\{x\}\cdot (xU)=U$, and for each $u \in U$, any representation $u=st$ with $s\in S$, $t\in T$ must have $s=x$ and $t=xu$.  Thus each $u\in U$ has a unique representation $u=st$ with $s\in S$, $t\in T$.
By~\Cref{thm:main_1}, $A(G;U)=A(G;S)A(G;T)=A(G;\{x\})A(G;Ux)$.
\end{proof}

\begin{cor}\label{cor:xUx=x_abelian}
Let $U\subseteq G\sm\{e\}$ be a symmetric subset of an abelian group $G$. 
If $G$ has an involution $x\notin U$, then $A(G;U)$ has a factorization. 
\end{cor}

When $(S,T,U)$ is factorable in $G$ and $x \in G - (S \cup T)$ is an element of order $2$ for which $xSx=S$ and $xTx=T$, then a similar argument to \Cref{lem:xUx=U} shows that $Sx$ and $xT$ will be symmetric subsets of $G$.  If this occurs, we will say that the factorizations $(S,T,U)$ and $(Sx,xT,U)$ are \defin{involution-equivalent}.  It is possible for two inequivalent factorizations to be involution-equivalent.  

\begin{exa}
Let $D_{10}=\langle r,s \mid r^{5}=e,\ s^{2}=e,\ srs=r^{-1}\rangle$.
We set
\[
S\coloneqq \{r,r^{4}\},\qquad T\coloneqq \{r^{2},r^{3}\},\qquad
U\coloneqq \{r,r^{2},r^{3},r^{4}\}=\langle r\rangle\setminus\{e\}.
\]
One can compute the products and get the following table:
\[
\begin{array}{c|cc}
st & t=r^{2} & t=r^{3}\\ \hline
s=r & r^{3} & r^{4}\\
s=r^{4} & r & r^{2}
\end{array}
\]
Hence $ST=\{r,r^{2},r^{3},r^{4}\}=U$, and each $u\in U$ occurs \emph{exactly once} as $u=st$ with
$(s,t)\in S\times T$. Therefore, by \Cref{thm:main_1},
we infer that $(S,T,U)$ is factorable in $D_{10}$.
Next let $x\coloneqq  s$ (an involution). Conjugation by $s$ inverts rotations gives us
$s r^{k} s=r^{-k}$ for $k\in\mathbb Z$.
Hence $xSx=S$ and $xTx=T$.
Next we define
\[
S'\coloneqq Sx=\{rs,\ r^{4}s\},\qquad T'\coloneqq xT=\{sr^{2},\ sr^{3}\}.
\]
Then $S'$ and $T'$ consist only of reflections, and for any $s_0\in S$, $t_0\in T$ we have $(s_0x)(xt_0)=s_0(x^{2})t_0=s_0t_0$.
Thus the map $(s_0,t_0)\mapsto (s_0x,xt_0)$ is a bijection $S\times T\to S'\times T'$ that preserves
the product, so $U$ has unique representations from $S'\times T'$ as well. Hence $(S',T',U)$ is factorable. By definition, $(S,T,U)$ and $(S',T',U)$ are
\emph{involution-equivalent} via $x=s$.
We note that for odd $n$, every automorphism of $D_{2n}$ maps rotations to rotations and reflections to reflections. In particular, no automorphism can send a set of
\emph{rotations} (like $S$ and $T$) to a set of \emph{reflections} (like $S'$ and $T'$).
Therefore there is no $\varphi\in\Aut(D_{10})$ with $(S',T',U)=(\varphi(S),\varphi(T),\varphi(U))$,
so the two factorizations are \emph{inequivalent}.
\end{exa}

A natural question is whether every Cayley graph of a finite group admits a factorization.
The answer is no. In fact, Maghsoudi et al.~\cite{maghsoudi2023matrix} proved that the complete
graph $K_n$ admits a factorization if and only if $n \equiv 1 \pmod{4}$ (equivalently, $n=4k+1$ for some $k\in\mathbb{N}$).
Since $K_{|G|} = A(G; G\sm\{e\})$, it follows that this Cayley graph is factorable precisely when
$|G|\equiv 1 \pmod{4}$, and is not factorable otherwise.
In the next theorem, we show that for every finite group $G$ there exists a generating set $S\subseteq G$
(with $e\notin S$ and $S=S^{-1}$) such that $A(G;S)$ does not admit a factorization.

\begin{thm}
Let $G$ be a finite group.
Then there exists a symmetric subset $U\subseteq G\sm\{e\}$ such that $\cay(G;U)$ has no factorization.
\end{thm}

\begin{proof}
If $|G|\not\equiv 1 \pmod{4}$, then $A(G;G\sm\{e\})$ admits no factorization by \cite[Theorem 1.2]{maghsoudi2023matrix}.
So we can assume that $|G|\equiv 1 \pmod{4}$.
Let $g$ be an element of a prime order $p$, where $p>2$.
Let $U=\{g,g^{-1}\}$.
If  $A(G;U)=A(G;S)\,A(G;T)$ with $S,T\subseteq G\sm\{e\}$ symmetric, then by \Cref{thm:main_1} and then by \Cref{lem:product},
we have $|U|=|S|\,|T|$.
Since $|U|$ is prime, one of $|S|$ or $|T|$ must be $1$. Suppose $|S|=1$. Then $S=\{s\}$, and because $S$ is symmetric and $e\notin S$, the element $s$ is an involution ($s=s^{-1}\neq e$).
This means that $2$ divides the order of $G$. But the order of $G$ is an odd number, which yields a contradiction.
\end{proof}

It follows from \cite[Thm.~3]{decaen}, that if $G$ has an elementary abelian quotient $C_p^m$,
then any near–factorization enforces congruences on $|S|,|T|$ modulo powers of $p$ that fail in many cases (e.g., for most abelian $p$–groups of sufficient rank); hence $A(G;G\sm\{e\})$ does not factor in those cases.

A natural question is whether every group $G$ has a symmetric subset $U$ of $G\sm \{e\}$ such that $\cay(G;U)$ admits a factorization.
In the following, we give an affirmative answer to the question.

\begin{thm}
Let $G$ be a finite group of order at least $4$.
Then there exists a symmetric subset $U\subseteq G\sm \{e\}$  such that $\cay(G;U)$ admits a factorization.
\end{thm}

\begin{proof}
We split the proof into the two cases separately.
\begin{itemize}
\item $G$ contains an involution.
Assume $s\in G$ with $s^2=e$ and $s\neq e$.
Set
\[
U\coloneqq G\sm\{e,s\}, \qquad S\coloneqq \{s\}, \qquad T\coloneqq sU=\{su\mid \ u\in U\}.
\]
Then $U$ is symmetric and $e\notin U$. Also $S$ and $T$ are symmetric.
For each $u\in U$ we have the \emph{unique} representation
\[
u \;=\; s\cdot (su)\qquad\text{with }s\in S,\ su\in T.
\]
Conversely, the only elements not in $U$ are $e$ and $s$: neither has a representation $st$
with $s\in S$, $t\in T$ (for $e$ we would need $t=s\notin T$.
For $s$ we would need $t=e\notin T$).
Therefore it follows from \Cref{thm:main_1} that $A(G;U)=A(G;S)A(G;T)$.
\item  $G$ has no involution (so $|G|$ is odd).
\begin{itemize}
\item $G$ is cyclic.
Write $G=\langle a\rangle$. We note that the order of $a$ should be at least $5$. 
Next we define
\[
S\coloneqq \{a,a^{-1}\},\qquad T\coloneqq \{a^{2},a^{-2}\},\qquad
U\coloneqq \{a,a^{-1},a^{3},a^{-3}\}.
\]
Then $S,T,U$ are symmetric and $e\notin S\cup T\cup U$.
Each $u\in U$ is obtained \emph{uniquely} as $u=st$ with $s\in S$, $t\in T$:
\[
a = a^{-1}\cdot a^{2},\quad a^{-1}=a\cdot a^{-2},\quad
a^{3}=a\cdot a^{2},\quad a^{-3}=a^{-1}\cdot a^{-2}.
\]
No other element of $G$ is represented. Thus $A(G;U)=A(G;S)A(G;T)$.
\item  $G$ is not cyclic.
If $G$ has an element $a$ of order at least $5$, then we again set
\[
S\coloneqq \{a,a^{-1}\},\qquad T\coloneqq \{a^{2},a^{-2}\},\qquad
U\coloneqq \{a,a^{-1},a^{3},a^{-3}\}.
\]
So we can assume that every element has order exactly $3$.
Since $|G|\geq 5$, there are distinct elements $a,b \in G\sm \{e\}$.
Then we define 
\[
S\coloneqq \{a,a^{-1}\},\qquad T\coloneqq \{b,b^{-1}\},\qquad
U\coloneqq \{ab,ab^{-1},a^{-1}b,a^{-1}b^{-1}\}.
\]
Then $S,T,U$ are symmetric and $e\notin S\cup T\cup U$.
In addition one can see that each $u\in U$ is obtained \emph{uniquely} as $u=st$ with $s\in S$, $t\in T$.\qedhere
\end{itemize}
\end{itemize}
\end{proof}

\begin{lem}{\rm\cite[Theorem 4]{Manjunatha}}\label{odd-edges}
Let $X, Y, Z$ be simple\footnote{no loops, and no multiple edge} graphs such that $A(X)=A(Y)A(Z)$, where $A(K)$ is the adjacency matrix of a graph $K$.
Then $ X $ has an even number of edges.
\end{lem}

\begin{lem}
Let $U\subseteq G\sm\{e\}$ be a symmetric subset of a group of order $n$. 
Suppose the Cayley graph $X=\cay(G;U)$ admits a factorization $A(X)=A(Y)A(Z)$ into adjacency matrices of simple graphs $Y,Z$ on the same vertex set.
If $n=2(2t+1)$, then $U$ contains an even number of involutions.
\end{lem}

\begin{proof}
By \Cref{odd-edges}, $X$ has an even number of edges. Since $X$ is $|U|$-regular on $n$ vertices,
\[
|E(X)|=\frac{n\,|U|}{2}\ \ \text{is even}\ \ \Longleftrightarrow\ \ n\,|U|\equiv 0 \pmod 4.
\tag{$\ast$}
\]
Then $n=2(2t+1)$, so $(\ast)$ gives $2|U|\equiv 0\pmod 4$, hence $|U|$ is even.
Write $U=I\sqcup P$, where $I$ is the set of involutions in $U$, and $P$ consists of
non–self-inverse elements. Because $U=U^{-1}$ and $e\notin U$, the elements of $P$ split into
disjoint 2-element inverse pairs $\{g,g^{-1}\}$, so $|P|$ is even. Therefore
$|U|=|I|+|P|\equiv |I|\pmod 2$. Since $|U|$ is even, $|I|$ is even.
\end{proof}

\begin{lem}\label{thm:transversal}
Let $G$ be a finite group and let $U,S,T\subseteq G$.
Assume that $A(G;U)=A(G;S)\,A(G;T)$.
Let $H\le G$ and let $G/H$ be the set of right cosets.
Then for every $y\in G$ we have
\[
|U\cap yH|
\;=\;
\sum_{x\in G/H} |S\cap Hx|\cdot |T\cap x^{-1}yH|.
\]
\end{lem}
\begin{proof}
Let $N(g)\coloneqq \bigl|\{(s,t)\in S\times T\mid  st=g\}\bigr|$.
The identity $A(G;U)=A(G;S)A(G;T)$ is equivalent to saying that
$N(g)=1$ for $g\in U$ and $N(g)=0$ for $g\notin U$ (i.e., $U=ST$ with
unique representations and $e\notin U$).
Fix $y\in G$. Summing $N(g)$ over the coset $yH$ gives
\[
|U\cap yH|=\sum_{g\in yH} N(g)
=\sum_{g\in yH}\ \sum_{\substack{x\in G/H\\ s\in S\cap Hx}}\ \#\{t\in T\mid  st=g\}.
\]
If $s\in Hx$, then $st\in yH$ holds iff $t\in x^{-1}yH$.
Thus the inner count is $|T\cap x^{-1}yH|$, and summing over $s\in S\cap Hx$
yields the claimed formula.
\end{proof}

\section{Cayley Factorization on Abelian Groups}

In this section, we investigate how Cayley graphs of abelian groups can be factorized.
Throughout, we use additive notation for the group operation, with $0$ denoting the identity element.

\begin{defn}
Let $(G,+)$ be a (finite) abelian group and let $f,g \colon G \to \mathbb{C}$ be functions.
The \defin{convolution} $f * g : G \to \mathbb{C}$ is defined by
\[
(f * g)(x) \;=\; \sum_{y \in G} f(y)\,g(x-y)
\qquad\text{for all } x\in G.
\]
\end{defn}

\begin{thm}\label{thm:circulant-conv-1}
Let $G$ be a finite abelian group with symmetric subsets $S,T\subseteq G\sm  \{0\}$.
Then we have the following:
$A(G;S)_{i,j}=\mathbf 1_S(g_j-g_i)$, and
$A(G;T)_{i,j}=\mathbf 1_T(g_j-g_i)$,
where $\mathbf 1_X$ is the indicator function of $X$.
Furthermore for all $1\le i,j\le n$,
\begin{equation}\label{eq:conv-entry}
(A(G;S)A(G;T))_{i,j}
\;=\;(\mathbf 1_S * \mathbf 1_T)(g_j-g_i)
\;=\;N_{S,T}(g_j-g_i),
\end{equation}
where $N_{S,T}(u)\coloneqq \bigl|\{(s,t)\in S\times T\mid \ s+t=u\}\bigr|$ is the number of
representations of $u$ as a sum from $S$ and $T$.
\end{thm}

\begin{proof}
Write $A_S\coloneqq A(G;S)$ and $A_T\coloneqq A(G;T)$. For $1\le i,j\le n$,
\[
(A_SA_T)_{i,j}
=\sum_{k=1}^n (A_S)_{i,k}\,(A_T)_{k,j}
=\sum_{k=1}^n \mathbf 1_S(g_k-g_i)\,\mathbf 1_T(g_j-g_k).
\]
Fix $i$ and set $x\coloneqq g_k-g_i$. Since $k\mapsto g_k$ is a bijection $[n]\!\to\! G$, the map $k\mapsto x$ is a bijection from $[n]$ onto $G$ (for each $x\in G$ there is a unique $k$ with $g_k=g_i+x$). Hence
\[
(A_SA_T)_{i,j}
=\sum_{x\in G} \mathbf 1_S(x)\,\mathbf 1_T\bigl((g_j-g_i)-x\bigr)
=(\mathbf 1_S * \mathbf 1_T)(g_j-g_i),
\]
where $(\mathbf 1_S * \mathbf 1_T)(y)\coloneqq \sum_{x\in G}\mathbf 1_S(x)\mathbf 1_T(y-x)$ is the group convolution on $G$.
Finally,
\[
(\mathbf 1_S * \mathbf 1_T)(u)
=\sum_{x\in G}\mathbf 1_S(x)\mathbf 1_T(u-x)
=\bigl|\{x\in S\mid \ u-x\in T\}\bigr|
=\bigl|\{(s,t)\in S\times T\mid \ s+t=u\}\bigr|
=N_{S,T}(u),
\]
so $(A_SA_T)_{i,j}=N_{S,T}(g_j-g_i)$.
\end{proof}

The study of factorizations in abelian groups is closely linked to the concept of Sidon pairs.
Let $ G $ be an abelian group (written additively).  
Two subsets $ S, T \subseteq G $ form a \defin{Sidon pair} if every sum $ s + t $ with $ s \in S $ and $ t \in T $ has a unique representation, that is,
\[
s_1 + t_1 = s_2 + t_2 \;\Longrightarrow\; (s_1, t_1) = (s_2, t_2).
\]
Equivalently, the multiset
$S + T \coloneqq \{\, s + t \mid s \in S,\, t \in T \,\}$
contains no repeated elements.

\begin{lem}\label{lem:sidon-pair}
Let $G$ be a finite abelian group (written additively). For non-empty subsets $S,T\subseteq G$, 
$N_{S,T}(u)\le 1$ for all $u\in G$
if and only if 
$(S-S)\cap(T-T)=\{0\}$,
where $$N_{S,T}(u)\coloneqq \bigl|\{(s,t)\in S\times T\mid \ s+t=u\}\bigr|.$$
\end{lem}

\begin{proof}
We first assume $N_{S,T}(u)\le 1$ for every $u\in G$. We show $(S-S)\cap(T-T)\subseteq\{0\}$.
Suppose for contradiction that there exists $x\in(S-S)\cap(T-T)$ with $x\ne 0$. Then there are elements
$s,s'\in S$ and $t,t'\in T$ such that
$x=s-s'$ and $x=t-t'$.
Rearranging  gives
$s+t' \;=\; s'+t$.
Set $u\coloneqq s+t'=s'+t$. Then both pairs $(s,t')$ and $(s',t)$ lie in $S\times T$ and represent $u$.
Since $x\neq 0$, at least one of $s\ne s'$ or $t\ne t'$ holds, so these two pairs are distinct.
Hence $N_{S,T}(u)\ge 2$, contradicting the hypothesis. Therefore no such $x$ exists, and$(S-S)\cap(T-T)=\{0\}$.\\
Conversely assume $(S-S)\cap(T-T)=\{0\}$. Let $u\in G$ and suppose $u$ admits two (possibly equal) representations
$u=s+t=s'+t'$, where $s,s'\in S,\ t,t'\in T$.
Subtracting the two equalities gives
$s-s' \;=\; t'-t$.
Thus $s-s'\in (S-S)\cap(T-T)=\{0\}$, so $s=s'$ and consequently $t=t'$.
Therefore every $u$ has \emph{at most} one representation as $s+t$ with $s\in S$, $t\in T$.
This proves $N_{S,T}(u)\le 1$ for all $u$, as required.
\end{proof}

\begin{cor}
Let $G$ be an abelian group.
If $(S,T,U)$ is factorable in $G$, then $$(S-S)\cap(T-T)=\{0\}.$$ 
\end{cor}

\begin{proof}
Suppose $A(G;U)=A(G;S)A(G;T)$.
It follows from \Cref{thm:circulant-conv-1}, we know that the adjacency matrices satisfies the following:
\[
(A(G;S)A(G;T))_{i,j}
= (\mathbf 1_S * \mathbf 1_T)(\,g_j-g_i\,)
= N_{S,T}(g_j-g_i),
\]
so equality with $A(G;U)$ (whose entries are $0/1$) implies $N_{S,T}(u)\in\{0,1\}$ for every $u$.
Applying the equivalence just proved yields $(S-S)\cap(T-T)=\{0\}$.
\end{proof}

\begin{thm}
Let $G$ be a finite abelian group of order $n$, and define
\[
d^*(G)\coloneqq \max\{\,d:\ \exists\ \text{symmetric }S,T\subseteq G\sm \{0\}\ \text{with }|S|=|T|=d
\ \text{and }N_{S,T}(u)\le 1\ \forall u\in G\,\}.
\]
Then $d^*(G)\le \lfloor \sqrt{n}\rfloor$. 
\end{thm}

\begin{proof}
For the upper bound, consider $\mu:S\times T\to G$ given by $\mu(s,t)=s+t$. If $\mu(s,t)=\mu(s',t')$, then $s+t=s'+t'$, hence $N_{S,T}(s+t)\ge 2$ unless $(s,t)=(s',t')$. Since $N_{S,T}(u)\le 1$ for all $u$, the map $\mu$ is injective. Thus
\[
d^2=|S||T|=|S\times T|\le |G|=n,
\]
so $d\le \lfloor \sqrt{n}\rfloor$.
\end{proof}

\begin{cor}
Let $G$ be a finite abelian group such that $G\cong H\oplus K$ with $|H|,|K|\ge2$, then
$d^*(G)\ge \min\{|H|,|K|\}-1$. 
In particular, $d^*(G)\in \Theta(\sqrt{n})$ for all families of abelian groups that admit such balanced decompositions.
\end{cor}
\subsection{Cayley Factorization of Cyclic Groups}

A \defin{circulant matrix} of order $n$ is a matrix $A = (a_{ij})$ such that 
$a_{ij} = c_{(j-i)\bmod n}$ for some vector $(c_0, c_1, \dots, c_{n-1}) \in \mathbb{C}^n$.

\[
\mathrm{Circ}(c_0, c_1, \dots, c_{n-1}) =
\begin{pmatrix}
c_0 & c_1 & c_2 & \cdots & c_{n-1} \\
c_{n-1} & c_0 & c_1 & \cdots & c_{n-2} \\
c_{n-2} & c_{n-1} & c_0 & \cdots & c_{n-3} \\
\vdots & \vdots & \vdots & \ddots & \vdots \\
c_1 & c_2 & c_3 & \cdots & c_0
\end{pmatrix}.
\]

A graph $ G $ is called \defin{circulant} if its adjacency matrix $ A(G) $ is a circulant matrix.  
It is straightforward to see that the Cayley graphs of finite cyclic groups are precisely the circulant graphs.  
Maghsoudi et~al.~\cite{maghsoudi2023matrix} proved that a certain family of circulant graphs admits a factorization.  
More precisely, they showed that 
$\cay(\Z_{2n};\{\pm n, \pm (n-1)\})$
can be factorized into 
$\cay(\Z_{2n};\{\pm n\})$ and $\cay(\Z_{2n};\{\pm 1\})$.

\begin{op}
Characterize all circulant graphs that admit a factorization up to equivalence.  
The graphs need not necessarily be Cayley graphs.
\end{op}

\begin{defn}
Let $U\subseteq \Z_n$.
The \defin{mask polynomial} of $U$ is
\[
F_U(X)\;=\;\sum_{u\in U} X^{u}\ \in\ \Z[X]/(X^{n}-1).
\]
Equivalently, for a primitive $n$th root of unity $\zeta_n$, one may view $F_U(\zeta_n)=\sum_{u\in U}\zeta_n^{u}$.
\end{defn}

\begin{thm}\label{thm:conv}
Let $U,S,T\subseteq \Z_n\sm\{0\}$ be symmetric sets.
The following are equivalent:
\begin{enumerate}
    \item $(S,T,U)$ is factorable in $\Z_n$.
    \item $\mathbf 1_U=\mathbf 1_S * \mathbf 1_T$.
    \item $F_U(x)\equiv F_S(x)\,F_T(x)\pmod{x^n-1}$.
\end{enumerate}
\end{thm}

\begin{proof}
Let $S,T\subseteq \Z_n\sm\{0\}$ be symmetric sets. Consider the adjacency matrices
$A_S \coloneqq A(\Z_n;S)$, and $A_T \coloneqq A(\Z_n;T)$.
The $(i,j)$-entry of the product is
\begin{align*}
(A_S A_T)_{i,j}
  &= \sum_{k\in \Z_n} (A_S)_{i,k}\,(A_T)_{k,j}
   \;=\; \sum_{k\in \Z_n} \mathbf{1}_S(k-i)\,\mathbf{1}_T(j-k).
\end{align*}
Let $m=k-i$. Then $k=i+m$ and $j-k=j-i-m$, so
\begin{align*}
(A_S A_T)_{i,j}
  &= \sum_{m\in \Z_n} \mathbf{1}_S(m)\,\mathbf{1}_T\bigl(j-i-m\bigr)
   \;=\; (\mathbf{1}_S * \mathbf{1}_T)(\,j-i\,),
\end{align*}
where $*$ denotes convolution on $\Z_n$.
Hence $A_SA_T$ is a circulant matrix whose first row is $\mathbf{1}_S * \mathbf{1}_T$.
We now suppose
$A(\Z_n;U) \;=\; A(\Z_n;S)\,A(\Z_n;T)$.
Since both sides are circulant, this holds if and only if their first rows agree:
$\mathbf{1}_U \;=\; \mathbf{1}_S * \mathbf{1}_T$.
Moreover, because $U,S,T\subseteq \Z_n\sm\{0\}$, we have
\[
\mathbf{1}_U(0)=\mathbf{1}_S(0)=\mathbf{1}_T(0)=0,
\]
as required for loopless Cayley graphs.

\medskip
Next we define the mask polynomials
\[
F_U(X)=\sum_{u\in U} X^{u},\qquad
F_S(X)=\sum_{s\in S} X^{s},\qquad
F_T(X)=\sum_{t\in T} X^{t}\ \ \in \ \Z[X]/(X^n-1).
\]
Their product modulo $X^n-1$ is
\begin{align*}
F_S(X)F_T(X)
  &= \sum_{s\in S}\sum_{t\in T} X^{s+t}
   \equiv \sum_{k\in \Z_n}
        \Bigl(\sum_{\substack{s\in S,\,t\in T\\ s+t\equiv k \,(\mathrm{mod}\, n)}} 1\Bigr) X^{k}
   \;=\; \sum_{k\in \Z_n} (\mathbf{1}_S * \mathbf{1}_T)(k)\,X^{k} \pmod{X^n-1}.
\end{align*}
Therefore, we have $F_S(X)F_T(X)\equiv F_U(X)\pmod{X^n-1}$ if and only if $
\mathbf{1}_S * \mathbf{1}_T \;=\; \mathbf{1}_U$.
Since $\mathbf{1}_U$ is $0\!-\!1$–valued, the coefficients of $F_S(X)F_T(X)\bmod (X^n-1)$ are in $\{0,1\}$.
Conversely, if the product has $0\!-\!1$ coefficients, it defines a set $U$ with $\mathbf{1}_U=\mathbf{1}_S * \mathbf{1}_T$.
\end{proof}

\begin{table}[H]
  \centering
  \caption{Ordinary circulant Cayley factorizations on $\Z_n$.
  In all rows we assume $0\notin S,T,U$, $S=-S$, $T=-T$ (so $U=-U$), and the
  \emph{uniqueness} condition $(S-S)\cap (T-T)=\{0\}$ (equivalently $N_{S,T}(u)\in\{0,1\}$ for all $u$), ensuring $A(\Z_n;U)=A(\Z_n;S)A(\Z_n;T)$ is a $0$–$1$ circulant (simple, undirected).}
  \label{tab:STU-Zn}
  \rowcolors{3}{gray!25}{white}
  \setlength{\tabcolsep}{8pt}
  \renewcommand{\arraystretch}{1.12}
  \begin{tabular}{@{}p{0.34\textwidth}p{0.18\textwidth}p{0.18\textwidth}p{0.22\textwidth}@{}}
    \toprule
    \textbf{Conditions} & \textbf{$S$} & \textbf{$T$} & \textbf{$U$} \\
    \midrule \midrule

    $g\in \Z_n^\times$ &$gS$ & $gT$ & $gU$ \\[2pt]

     $n$ even and $U\subseteq \Z_n\sm \{0,\tfrac n2\}$ &
    $\{\tfrac n2\}$ & $U-\tfrac n2$ & $U$ \\[2pt]

    $\operatorname{ord}(d)\ge 5$ &
    $\{\pm d\}$ & $\{\pm 2d\}$ & $\{\pm d,\ \pm 3d\}$ \\[2pt]

    $I,J\subseteq\{1,\dots,\lfloor\frac{n-1}{2}\rfloor\}$ so that the multiset
    $\{\pm(i\pm j)\mid \ i\in I,\ j\in J\}$ has no repetitions modulo $n$, avoids $0$, and (if $n$ is even) also avoids $\tfrac n2$ &
    $\{\pm i\mid \ i\in I\}$ & $\{\pm j\mid\ j\in J\}$ & $\{\pm(i\pm j)\mid \ i\in I,\ j\in J\}$ \\[2pt]
    \bottomrule
  \end{tabular}
  \vspace{0.6ex}
\end{table}
\begin{lem}\label{lem:antipode-augmentation}
Let $n$ be even and $a\coloneqq n/2$.
Assume that $(S_0, T , U_0 )$ is factorable in $\Z_n$,
where
$S_0=\{\pm i\mid \ i\in I\}$, $T=\{\pm j\mid\ j\in J\}$,
$U_0=\{\pm(i\pm j)\mid \ i\in I,\ j\in J\}$.
If $a+T$ is disjoint from $U_0$, i.e.
$\{\,\pm(a+j)\mid \ j\in J\,\}\ \cap\ U_0\;=\;\varnothing$,
then $(S, T , U )$ is also factorable in $\Z_n$, where
\[
S\;=\;S_0\cup\{a\},\qquad T\;=\;\{\pm j\mid\ j\in J\},\qquad
U\;=\;U_0\ \cup\ \{\pm(a+j)\mid\ j\in J\}.
\]
\end{lem}

\begin{proof}
By~\Cref{thm:circulant-conv-1}, we have $\mathbf 1_{U_0}=\mathbf 1_{S_0}*\mathbf 1_T$.
Since $S=S_0\cup\{a\}$, we have
\[
\mathbf 1_S*\mathbf 1_T
=\bigl(\mathbf 1_{S_0}+\mathbf 1_{\{a\}}\bigr)*\mathbf 1_T
= \mathbf 1_{S_0}*\mathbf 1_T \;+\; \mathbf 1_{\{a\}}*\mathbf 1_T
= \mathbf 1_{U_0} \;+\; \mathbf 1_{a+T}.
\]
Because $a=-a$ and $T=-T$, the set $a+T$ equals $\{\pm(a+j): j\in J\}$ and is inverse‑closed.
Also $0\notin a+T$ since $a\notin T$. By the disjointness hypothesis $(a+T)\cap U_0=\varnothing$,
the sum $\mathbf 1_{U_0}+\mathbf 1_{a+T}$ is $0$–$1$, i.e.,
\[
\mathbf 1_S*\mathbf 1_T=\mathbf 1_{U_0\cup(a+T)}=\mathbf 1_U,
\]
which is equivalent to $A(\Z_n;U)=A(\Z_n;S)A(\Z_n;T)$. The sets $S$ and $T$ are inverse‑closed and exclude $0$,
so the factors and the product remain simple and undirected.
\end{proof}
\begin{nota}
Let $n=\prod_{i=1}^{k} p_i^{e_i}$. By the Chinese Remainder Theorem, there is a group isomorphism
\[
\begin{array}{rcl}
\varphi \colon \Z_n &\longrightarrow& \displaystyle\prod_{i=1}^{k} \Z_{p_i^{e_i}} \\[4pt]
x &\longmapsto& \bigl(x \bmod p_1^{e_1},\, \dots,\, x \bmod p_k^{e_k}\bigr)
\end{array}
\]

\end{nota}

\begin{lem}{\rm \cite{MR2215618}}\label{lem:prod_group_ring}
Let $R$ be a commutative ring and let $G,H$ be groups. There is a natural
isomorphism of $R$-algebras
\[
R[G\times H]\ \cong\ R[G]\otimes_{R} R[H].
\]
\end{lem}


\begin{thm}\label{thm:CRT}
Let $n=\prod_{i=1}^k p_i^{e_i}$ and let
$\varphi\colon \Z_n \;\xrightarrow{\ \cong\ }\; \prod_{i=1}^k \Z_{p_i^{e_i}}$
be the CRT isomorphism.
Assume that $U,S,T\subseteq \Z_n$ are of \emph{product form}:
there exist subsets $U_i,S_i,T_i\subseteq \Z_{p_i^{e_i}}$ such that
\[
U=\varphi^{-1}\Bigl(\prod_{i=1}^k U_i\Bigr),\qquad
S=\varphi^{-1}\Bigl(\prod_{i=1}^k S_i\Bigr),\qquad
T=\varphi^{-1}\Bigl(\prod_{i=1}^k T_i\Bigr).
\]
Then $(S,T,U)$ is factorable in $\Z_n$ if and only if, for every $i$,
$(S_i,T_i,U_i)$ is factorable in $\Z_{p_i^{e_i}}$.
\end{thm}
\begin{proof}
Write $G\coloneqq \prod_{i=1}^k \Z_{p_i^{e_i}}$ and view $\varphi$ as a relabeling
of vertices. Hence the identity in $\Z_n$ holds iff the corresponding
identity holds in $G$ for the sets $\prod U_i$, $\prod S_i$, $\prod T_i$.

Now use the general fact for direct products: if $G=G_1\times\cdots\times G_k$
and $X=\prod_{i=1}^k X_i$, then
\[
A(G;X)=A(G_1;X_1)\otimes \cdots \otimes A(G_k;X_k),
\]
because adjacency between $(g_1,\dots,g_k)$ and $(h_1,\dots,h_k)$ is equivalent to
$g_i^{-1}h_i\in X_i$ for every $i$.
Therefore
\[
A(G;\textstyle\prod U_i)
=
\bigotimes_{i=1}^k A(G_i;U_i),\quad
A(G;\textstyle\prod S_i)
=
\bigotimes_{i=1}^k A(G_i;S_i),\quad
A(G;\textstyle\prod T_i)
=
\bigotimes_{i=1}^k A(G_i;T_i).
\]
Using $(A_1\otimes\cdots\otimes A_k)(B_1\otimes\cdots\otimes B_k)
=(A_1B_1)\otimes\cdots\otimes(A_kB_k)$, we get that
$A(G;\prod U_i)=A(G;\prod S_i)A(G;\prod T_i)$ holds iff
$A(G_i;U_i)=A(G_i;S_i)A(G_i;T_i)$ holds for all $i$.
\end{proof}


\section{Cayley Factorization of Dihedral Groups}

Let $D_{2n}$ be the dihedral group of order $2n$ with presentation
\[
D_{2n}=\langle r,s \mid r^n=e,\ s^2=e,\ srs=r^{-1}\rangle.
\]
Its elements split into the rotations 
$R=\{e,r,\dots,r^{\,n-1}\}$
and the reflections (which are involutions as well)
$M=\{s,sr,\dots,sr^{\,n-1}\}$.
In this section, we explore the factorization of $D_{2n}$.
We first see an example of factorization of a Cayley graph of $D_{2n}$.

\begin{figure}[H]
    \centering
    \includegraphics[scale=0.8]{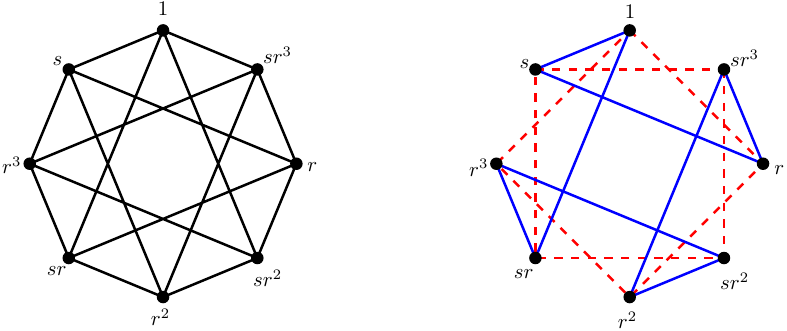}
    \caption{The Cayley graph $\cay(D_8,\{sr^3,s,sr^2,sr\})$ has a factorization to $\cay(D_{8};\{r,r^3\})$ and $\cay(D_8;\{s,sr\})$.}
    \label{fig:placeholder}
\end{figure}

\begin{thm}
Let $U$ be a symmetric subset of $D_{2n}\sm\{e\}$.
If the Cayley graph $\cay(D_{2n};U)$ admits a factorization
$(S,T,U)$ is factorable in $D_{2n}$
with $S,T\subseteq D_{2n}\sm\{e\}$ symmetric, then we have 
\begin{equation}\label{eq:UR-UM}
U\cap R=(S_R T_R)\ \bigsqcup\ (S_M T_M),\qquad
U\cap M=(S_R T_M)\ \bigsqcup\ (S_M T_R),
\end{equation}
\end{thm}

\begin{proof}
Assume $A(D_{2n};U)=A(D_{2n};S)A(D_{2n};T)$ with $S,T\subseteq D_{2n}\sm\{e\}$ symmetric.
Let $X$ be a subset of $D_{2n}$.
Then $X_R(\textit{resp.} X_M)$ denotes the set of rotations(\textit{resp.} reflections) of the set $X$.
Decompose $S=S_R\cup S_M$ and $T=T_R\cup T_M$ with $S_R,T_R\subseteq R$ and $S_M,T_M\subseteq M$.
Using the dihedral product rules,
\[
R\cdot R,\ M\cdot M\subseteq R,\qquad R\cdot M,\ M\cdot R\subseteq M,
\]
If $u\in U\cap R$, then $u=st$ for a unique $s\in S, t\in T$ by the factorization/uniqueness hypothesis. 
Since $u$ is a rotation, the above product rules force $(s,t)$ to be either (\text{rotation},\text{rotation}) or (\text{reflection},\text{reflection}). 
Hence $u\in S_R T_R$ or $u\in S_M T_M$.
We now show that $S_R T_R$ and $S_M T_M$ are disjoint.
Assume to the contrary that $s_1t_1=s_2t_2\in S_R T_R \cap S_M T_M$, where $s_1\in S_R, s_2\in S_M$ and $t_1\in T_R, t_2\in T_M$.
But it follows from \Cref{thm:main_1} that $s_1=s_2$ and $t_1=t_2$.
This means that $S_R\cap S_M\neq \varnothing$ and $T_R\cap T_M\neq \varnothing$ which violates the assumption that $R$ and $M$ are disjoint sets.
\end{proof}
It turns out that it is difficult to classify all subsets $U\subseteq D_{2n}\sm\{e\}$ for which the Cayley graph $\cay(D_{2n};U)$ admits a factorization. 
Accordingly, we restrict our attention to the case where $\cay(D_{2n};U)$ is connected.
Recall that $\cay(G;U)$ is connected if and only if $U$ generates $G$.
We make a list of generating sets $U$ of $D_{2n}$ such that the Cayley graph $\cay(D_{2n};U)$ admits a factorization, see \Cref{tab:condition-STU-dihedral}.

\begin{table}[H]
  \centering
  \caption{Factorable triples $(S,T,U)$ in $D_{2n}=\langle r,s\mid r^n=e,\ s^2=e,\ srs=r^{-1}\rangle$.
  In each row, pick the indicated parameters, take a reflection $x\in S$, and set $T=Ux$.}
  \label{tab:condition-STU-dihedral}
  \rowcolors{3}{gray!25}{white}
  \setlength{\tabcolsep}{10pt}
  \renewcommand{\arraystretch}{1.15}
  \begin{tabular}{@{}p{0.25\textwidth}p{0.05\textwidth}p{0.25\textwidth}p{0.25\textwidth}@{}}
    \toprule
    \textbf{Condition} & \textbf{$S$} & \textbf{$T$} & \textbf{$U$} \\
    \midrule \midrule

    $n$ odd;  &
    $\{\,s\,\}$ &
    $\{\,r^{-1},\,r\,\}$ &
    $\{\,sr,\,sr^{-1}\,\}$ \\

    all $n\ge3$; &
    $\{\,s\,\}$ &
    $\{\,rs,\,r^{-1}s,\,r^{-1},\,r\,\}$ &
    $\{\,r,\,r^{-1},\,sr,\,sr^{-1}\,\}$ \\

    all $n\ge3$;  &
    $\{\,s\,\}$ &
    $\{\,r^{-k}\mid \ 1\le k\le n-1\,\}$ &
    $\{\,sr^k\mid \ 1\le k\le n-1\,\}$ \\

    all $n\ge3$;  $\gcd(m,n)=1$; &
    $\{\,s\,\}$ &
    $\{\,r^{m}s,\,r^{-m}s,\,r^{-u},\,r^{u}\,\}$ &
    $\{\,r^{\pm m},\,sr^{\pm u}\,\}$ \\

    $n$ even;  &
    $\{\,s\,\}$ &
    $\{\,r^{-k}:\ k\ \text{odd}\}\cup\{rs,r^{-1}s\}$ &
    $\{\,sr^{k}\mid \ k\ \text{odd}\}\cup\{r,\,r^{-1}\}$ \\

    all $n\ge3$;  &
    $\{\,s\,\}$ &
    $\{\,rs,\,r^{-1}s,\,r^{-2},\,r^{2}\,\}$ &
    $\{\,r^{\pm1},\,sr^{\pm2}\,\}$ \\

    $n$ odd;  $\gcd(u,n)=1$; &
    $\{\,sr^{a}\,\}$ &
    $\{\,r^{-u},\,r^{u}\,\}$ &
    $\{\,sr^{a+u},\,sr^{a-u}\,\}$ \\

    all $n\ge3$; $\gcd(m,n)=1$ &
    $\{\,sr^{a}\,\}$ &
    $\{\,sr^{a-m},\,sr^{a+m}\,\}\ \cup\ \{\,r^{\mp u_i}\mid i=1,\dots,k\,\}$ &
    $\{r^{\pm m}\}\cup\{sr^{a\pm u_i}\mid i=1,\dots,k\}$ \\

    \bottomrule
  \end{tabular}
\end{table}

\begin{op}
Find all symmetric generating sets $U$ of $D_{2n}$ such that the Cayley graph $\cay(D_{2n};U)$ admits a factorization up to equivalence.
\end{op}

Recently, \citet{KreherPatersonStinson2025} investigated near-factorizations of the dihedral group 
$D_{2n}$ obtained from near-factorizations of the cyclic group.
We extend their results to a broader setting, focusing on factorizations of Cayley graphs by Cayley graphs.

\begin{defn}({\rm\citet{KreherPatersonStinson2025}})\label{def:Pecher}
Let $n\ge 3$ be odd. 
We define the isomorphism $\varphi$ and the function\footnote{$\psi$ is a bijection of sets, not a homomorphism} $\psi$ \[ \begin{array}{rcl} \varphi \colon \Z_{2n} &\longrightarrow& \Z_2 \times \Z_n \\[4pt] x &\longmapsto& \bigl(x \bmod 2,\, x \bmod n\bigr) \end{array} \qquad\text{and}\qquad \begin{array}{rcl} \psi \colon \Z_2 \times \Z_n &\longrightarrow& D_{2n} \\[4pt] (i,j) &\longmapsto& s^{\,i} r^{\,j} \end{array} \]
\begin{itemize}
\item For $X\subseteq \Z_{2n}$, define the \defin{forward P\^echer}
\[
\mathcal P(X)\ \coloneqq \ \psi\bigl(\varphi(X)\bigr)\ \subseteq\ D_{2n}.
\]
\item For $\widetilde X\subseteq D_{2n}$, define the \defin{inverse P\^echer}
\[
\mathcal P^{-1}(\widetilde X)\ \coloneqq \ \varphi^{-1}\!\bigl(\psi^{-1}(\widetilde X)\bigr)\ \subseteq\ \Z_{2n}.
\]
\end{itemize}
For tuples we apply $\mathcal P$ and $\mathcal P^{-1}$ componentwise, e.g.
$\mathcal P(S,T,U)=(\widehat S,\widehat T,\widehat U)$ with $\widehat S=\mathcal P(S)$, etc.
\end{defn}

\begin{nota}
For a set $X\subseteq\Z_{2n}$ write
$\widehat X\coloneqq \mathcal P(X)\subseteq D_{2n}$, and for $\widetilde X\subseteq D_{2n}$ write
$X\coloneqq \mathcal P^{-1}(\widetilde X)\subseteq\Z_{2n}$.
\end{nota}

\begin{thm}\label{thm:Z2n-to-D2n-factorable}
Let $n\ge 3$ be odd. 
Then the following holds:
\begin{enumerate}
\item If $(S,T,U)$ is factorable in $\Z_{2n}$,
then $(\widehat S,\widehat T,\widehat U)$ is factorable in $D_{2n}$.
Moreover, if $(S,T,U)$ is equivalent to $(A,B,C)$ in $\Z_{2n}$, then
$(\widehat S,\widehat T,\widehat U)$ is equivalent to $(\widehat A,\widehat B,\widehat C)$
in $D_{2n}$.
\item If $(\widetilde S,\widetilde T,\widetilde U)$ is factorable in $D_{2n}$, then
$(S,T,U)$ is factorable in $\Z_{2n}$. 
If, in addition,
$|\widetilde S|=k$, $|\widetilde T|=\ell$ satisfy
\[
\gcd\!\Bigl(n,\tfrac{k+1}{2}\Bigr)=\gcd\!\Bigl(n,\tfrac{\ell+1}{2}\Bigr)=1,
\]
then equivalence in $D_{2n}$ pulls back to equivalence in $\Z_{2n}$ in the sense that, whenever
$(\widetilde S,\widetilde T,\widetilde U)$ is equivalent to $(\widetilde A,\widetilde B,\widetilde C)$ in
$D_{2n}$, the triples
\[
\bigl(\varphi^{-1}\psi^{-1}(\widetilde S),\ \varphi^{-1}\psi^{-1}(\widetilde T),\ \varphi^{-1}\psi^{-1}(\widetilde U)\bigr)
\quad\text{and}\quad
(A,B,C)
\]
are equivalent in $\Z_{2n}$.
\end{enumerate}
\end{thm}

\begin{proof}
By~\Cref{thm:main_1}, we know that $(S,T,U)$ is factorable in $\Z_{2n}$ if and only if 
the “unique-sum” counts satisfy
\[
N_{S,T}(x)=
\begin{cases}
1,&x\in U,\\
0,&x\notin U,
\end{cases}
\qquad
N_{S,T}(x)\coloneqq |\{(s,t)\in S\times T\mid \ s+t=x\}|.
\]
Next we set the following sets
$S^*\coloneqq \varphi(S)$, $T^*\coloneqq \varphi(T)$, $U^*\coloneqq \varphi(U)\subseteq\Z_2\times\Z_n$.
We shall show that the same uniqueness holds for componentwise opearation in $\Z_2\times\Z_n$.
We note that every $(\alpha,\beta)\in U^*$ has a \emph{unique} representation
$(\alpha,\beta)=(i,k)+(i',\ell)$ with $(i,k)\in S^*$, $(i',\ell)\in T^*$, as we desired.\\
Let $\widehat S=\psi(S^*)$, $\widehat T=\psi(T^*)$, $\widehat U=\psi(U^*)$.
We show that the “unique-product” condition holds in $D_{2n}$.
We pick up an element $x\in D_{2n}\sm \{e\}$.
Then we have the following cases:
\begin{itemize}
\item If $x=r^j$ is a rotation, write $(0,j)=(i,k)+(i',\ell)$ in $\Z_2\times\Z_n$ with
$(i,k)\in S^*$, $(i',\ell)\in T^*$. If $i=i'=0$ then
$\psi(i,k)\psi(i',\ell)=r^{k}r^{\ell}=r^j$. If $i=i'=1$ then
$\psi(i,k)\psi(i',\ell)=(sr^{k})(sr^{\ell})=r^{\ell-k}$; by symmetry of $S^*,T^*$
we can replace $k$ by $-k$ to get $r^{(-k)+\ell}=r^j$. 
\item If $x=sr^j$ is a reflection, write $(1,j)=(i,k)+(i',\ell)$. If $(i,i')=(0,1)$ then
$\psi(i,k)\psi(i',\ell)=r^{k}(sr^{\ell})=sr^{k+\ell}=sr^j$. 
If $(i,i')=(1,0)$ then
$(sr^{k})r^{\ell}=sr^{k+\ell}=sr^j$.
\end{itemize}

Thus every non-identity element of $D_{2n}$ lies in $\widehat S\,\widehat T$.
Next we shall show the uniqueness.
Suppose $x=\widehat s_1\widehat t_1=\widehat s_2\widehat t_2$ with
$\widehat s_i\in\widehat S$, $\widehat t_{\ell}\in\widehat T$ for $\ell=1,2$.
Apply $\psi^{-1}$ to obtain two representations of either $(0,j)$ or $(1,j)$ in
$\Z_2\times\Z_n$ as $(i,k)+(i',\ell)$. Uniqueness in $\Z_2\times\Z_n$ forces
$(i,k)=(i',k')$ and $(i',\ell)=(i'',\ell')$, hence $\widehat s_1=\widehat s_2$ and
$\widehat t_1=\widehat t_2$. Therefore the product representation in $D_{2n}$ is unique.
It follows from~\Cref{thm:main_1} that the uniqueness of representations is equivalent to
$A(D_{2n};\widehat U)=A(D_{2n};\widehat S)\,A(D_{2n};\widehat T)$, as required.

Next we show the equivalence under automorphisms.
More precisely we show
if $(S,T,U)$ is equivalent to $(A,B,C)$ in $\Z_{2n}$ via an automorphism
$x\mapsto ux$ with $u\in\Z_{2n}^\times$, then 
$(\widehat S,\widehat T,\widehat U)$ is equivalent to $(\widehat A,\widehat B,\widehat C)$
in $D_{2n}$ via the automorphism $f\in\Aut(D_{2n})$ defined by $f(r)=r^{\,u\bmod n}$, $f(s)=s$.

If $(A,B,C)=(uS,uT,uU)$ with $u\in\Z_{2n}^\times$, then
$S^*\mapsto f(S^*)$, $T^*\mapsto f(T^*)$, $U^*\mapsto f(U^*)$ under
$f(i,j)=(i,\,u j)$ on $\Z_2\times\Z_n$. 
Next we consider the following automorphism of $D_{2n}$
\[
\begin{array}{rcl}
f_{u,0} \colon D_{2n} &\longrightarrow& D_{2n} \\[4pt]
r &\longmapsto& r^{\,u}, \\[2pt]
s &\longmapsto& s
\end{array}
\]
Then it is not hard to see that 
$\psi\circ f=f_{u,0}\circ \psi$
 which implies $(\widehat S,\widehat T,\widehat U)\mapsto
\bigl(f_{u,0}(\widehat S),\,f_{u,0}(\widehat T),\,f_{u,0}(\widehat U)\bigr)$,
i.e., equivalence is preserved. 

Let $(\widetilde S,\widetilde T,\widetilde U)$ is factorable in $D_{2n}$.
Then we set
$S^*=\psi^{-1}(\widetilde S)$, $T^*=\psi^{-1}(\widetilde T)$, $U^*=\psi^{-1}(\widetilde U)$
in $\Z_2\times\Z_n$ and then $S=\varphi^{-1}(S^*)$, etc.
The same two product patterns as above show that every nonzero $(\alpha,\beta)$
has a unique decomposition as a sum from $S^*,T^*$.
By applying $\varphi^{-1}$, we derive the uniqueness in $\Z_{2n}$, hence $A(\Z_{2n};U)=A(\Z_{2n};S)A(\Z_{2n};T)$.

We need to show that
$(\widetilde S,\widetilde T,\widetilde U)$ and $(\widetilde A,\widetilde B,\widetilde C)$
are equivalent if there exists $f\in\Aut(D_{2n})$ with
$(\widetilde A,\widetilde B,\widetilde C)=(f(\widetilde S),f(\widetilde T),f(\widetilde U))$.
For odd $n$, every automorphism has the form
\[
f_{u,v}:\quad r\mapsto r^{\,u},\qquad s\mapsto s\,r^{\,v},
\qquad (u\in\Z_n^{\times},\ v\in\Z_n).
\]
Conjugating by $\psi$ shows the induced action on $\Z_2\times\Z_n$ is the affine map
\[
F_{u,v}:\ (i,j)\longmapsto\bigl(i,\;u\,j+i\,v\bigr).
\]
Thus if $(\widetilde A,\widetilde B,\widetilde C)=\bigl(f_{u,v}(\widetilde S),f_{u,v}(\widetilde T),
f_{u,v}(\widetilde U)\bigr)$, then with $S^*=\psi^{-1}(\widetilde S)$, etc.,
\[
A^*=F_{u,v}(S^*),\qquad B^*=F_{u,v}(T^*),\qquad C^*=F_{u,v}(U^*).
\]
Applying $\varphi^{-1}$ gives the pulled–back triple $(A,B,C)\subseteq\Z_{2n}$.

Let us fix $k=|\widetilde S|$, $\ell=|\widetilde T|$, and let
$s\coloneqq |\widetilde S\cap s\langle r\rangle|$, $t\coloneqq |\widetilde T\cap s\langle r\rangle|$
be the counts of reflections in the two factors. 
It is not hard to see that 
\[
s=\frac{k+\varepsilon}{2},\qquad t=\frac{\ell-\varepsilon}{2}
\quad\text{for some }\varepsilon\in\{+1,-1\};
\]
in particular, one of $\frac{k+1}{2},\frac{k-1}{2}$ equals $s$, and one of
$\frac{\ell+1}{2},\frac{\ell-1}{2}$ equals $t$.

Under $F_{u,v}$, the rotation indices (the $i=0$ slice of $S^*$) are multiplied by $u$,
while the reflection indices (the $i=1$ slice) are sent to $u\,M_S+v$, where
$M_S\coloneqq \{\,j:\ (1,j)\in S^*\,\}\subseteq\Z_n$.
If the pulled–back equivalence on $\Z_{2n}$ is to be a pure multiplier $x\mapsto u'x$,
then necessarily $v\equiv 0\pmod n$; otherwise the reflection index set would be
translated by a nonzero $v$. To force $v\equiv 0$, observe that
\[
u\,M_S+v\ =\ M_A\quad\Longrightarrow\quad M_S+v\ =\ u^{-1}M_A,
\]
so $M_S$ is invariant under translation by $v$. Hence $M_S$ is a union of cosets of
the subgroup $\langle v\rangle\le\Z_n$, and therefore $|M_S|$ is a multiple of
$n/\gcd(n,v)$. But $|M_S|=s$ equals either $\tfrac{k+1}{2}$ or $\tfrac{k-1}{2}$.
If, in particular, $\gcd\!\bigl(n,\tfrac{k+1}{2}\bigr)=1$, then the only way a nonempty
subset of size $s$ can be a union of cosets of $\langle v\rangle$ is $v\equiv 0\pmod n$.
Applying the same argument to $T$ and the hypothesis
$\gcd\!\bigl(n,\tfrac{\ell+1}{2}\bigr)=1$ yields $v\equiv 0\pmod n$ as well.

Consequently the affine map reduces to $F_{u,0}(i,j)=(i,uj)$, which after $\varphi^{-1}$
becomes the multiplier $x\mapsto u'x$ on $\Z_{2n}$ (with $u'\equiv u\pmod n$, $u'\equiv 1\pmod 2$).
Thus the pulled–back triples $(A,B,C)$ and $(S,T,U)$ are equivalent in $\Z_{2n}$ via
an automorphism $x\mapsto u'x$.
\end{proof}

\begin{defn}
Let $n\ge 3$ be odd.
A subset $\widetilde X\subseteq D_{2n}$ is \defin{strongly symmetric} if, writing
$R=\{r^j:\ j\in\Z_n\}$ and $M=\{sr^j:\ j\in\Z_n\}$, then we have
\[
r^j\in \widetilde X\cap R\ \Longleftrightarrow\ r^{-j}\in \widetilde X\cap R,\qquad
sr^j\in \widetilde X\cap M\ \Longleftrightarrow\ sr^{-j}\in \widetilde X\cap M.
\]
\end{defn}

\begin{thm}({\rm \citet[Theorem 8]{KreherPatersonStinson2025}})\label{thm:KPS8}
Let $n$ be odd. If $(\widetilde S,\widetilde T, D_{2n}\sm \{e\})$ is a \emph{strongly symmetric}
near–factorization of $D_{2n}$, then
$\bigl(\mathcal P^{-1}(\widetilde S), \mathcal P^{-1}(\widetilde T),\Z_{2n}\sm \{0\}\bigr)$
is a \emph{symmetric} near–factorization of $\Z_{2n}$.
\end{thm}

\begin{cor}\label{cor:ours-implies-KPS8}
\Cref{thm:Z2n-to-D2n-factorable} $\Rightarrow$~\Cref{thm:KPS8}
\end{cor}

\begin{proof}
We invoke~\Cref{thm:Z2n-to-D2n-factorable}(b) with 
$\widetilde U=D_{2n}\sm \{e\}$ and a strongly symmetric pair
$(\widetilde S,\widetilde T)$ of $D_{2n}$.
Then $(S,T,U)=\bigl(\mathcal P^{-1}(\widetilde S),\,\mathcal P^{-1}(\widetilde T),\,
\varphi^{-1}(\psi^{-1}(\widetilde U))\bigr)$ is factorable in $\Z_{2n}$.
Since $\widetilde U=D_{2n}\sm \{e\}$, we have
$U=\varphi^{-1}(\psi^{-1}(\widetilde U))=\Z_{2n}\sm \{0\}$, so $(S,T)$ is a near–factorization of
$\Z_{2n}$. Moreover, the strong symmetry of $(\widetilde S,\widetilde T)$ implies
$S=-S$ and $T=-T$, i.e.\ $(S,T)$ is \emph{symmetric}.
\end{proof}

\begin{lem}\label{lem:symmetric-to-strongly-symmetric}
Let $n\ge3$ be odd and let $X\subseteq \mathbb Z_{2n}$ be symmetric.
Then $\widehat X\coloneqq\mathcal P(X)\subseteq D_{2n}$ is strongly symmetric.
Moreover, if $0\notin X$, then $e\notin\widehat X$.
\end{lem}

\begin{proof}
For $x\in\mathbb Z_{2n}$ write $\varphi(x)=(i,j)\in\mathbb Z_2\times\mathbb Z_n$ and $\mathcal P(x)=\psi(i,j)=s^i r^j$.
One can see that
$\varphi(-x) = (-x\bmod2,\,-x\bmod n)= (i,\,-j)$, as $-1= 1 \bmod2$.
Hence, we obtain the following
\[
\mathcal P(-x) = \psi(i,-j) = s^i r^{-j}.
\] 
Suppose $r^j\in\widehat X\cap R$. Then there is some even $x\in X$ with
  $\varphi(x)=(0,j)$, so $\mathcal P(x)=r^j$.  
  Symmetry of $X$ gives $-x\in X$, and then we get the following
  \[
  \mathcal P(-x)=s^0 r^{-j} = r^{-j}\in\widehat X\cap R.
  \]
  The converse is identical with $j$ replaced by $-j$, so
  $r^j\in\widehat X\cap R$ if and only if  $r^{-j}\in\widehat X\cap R$.
We next suppose $sr^j\in\widehat X\cap M$. Then there is some odd $x\in X$ with
  $\varphi(x)=(1,j)$, so $\mathcal P(x)=sr^j$.  
  Again by $X=-X$, we have $-x\in X$, and
  $\mathcal P(-x)=s^1 r^{-j} = s r^{-j}\in\widehat X\cap M$.
  As before, the converse follows by replacing $j$ with $-j$, giving
  $sr^j\in\widehat X\cap M$ if and only if $sr^{-j}\in\widehat X\cap M$.

Thus $\widehat X$ is strongly symmetric.
Finally, $0\in\mathbb Z_{2n}$ is the unique element with
$\varphi(0)=(0,0)$ and hence $\mathcal P(0)=r^0=e$.
So if $0\notin X$, no element of $X$ maps to $e$, i.e.\ $e\notin\widehat X$.
\end{proof}

\begin{cor}\label{cor:symmetric-to-strongly-symmetric-factorable}
Let $n\ge3$ be odd and let $(S,T,U)$ be factorable in $\mathbb Z_{2n}$.
Assume that $S,T,U\subseteq\mathbb Z_{2n}\sm\{0\}$ are symmetric with 
$\widehat S\coloneqq\mathcal P(S),
\widehat T\coloneqq\mathcal P(T)$, and
$\widehat U\coloneqq\mathcal P(U)$.
Then the following holds:
\begin{enumerate}
\item $(\widehat S,\widehat T,\widehat U)$ is factorable in $D_{2n}$; and
\item each of $\widehat S,\widehat T,\widehat U$ is strongly symmetric.
In particular, $(\widehat S,\widehat T,\widehat U)$ is a strongly symmetric matrix product factorization in $D_{2n}$.
\end{enumerate}
Moreover, if $(S,T,U)$ and $(A,B,C)$ are equivalent symmetric factorizations in $\mathbb Z_{2n}$,
then $(\widehat S,\widehat T,\widehat U)$ and $(\widehat A,\widehat B,\widehat C)$ are equivalent strongly symmetric factorizations in $D_{2n}$.
\end{cor}


\section{Further Research}
We close the paper, with a list of different directions for further research.
\begin{enumerate}
\item In this paper we investigated \emph{Cayley-by-Cayley factorizations}
in which each factor is itself the adjacency matrix of a Cayley graph on the same group~$G$.
Because Cayley graphs form the $1$-class of the standard group association scheme, this viewpoint naturally extends to the broader setting of \emph{association schemes}.
Accordingly, the study of Cayley-by-Cayley factorizations leads to the following general problem:
\begin{op}
Given an association scheme $\mathfrak X=(X,\{A_0,\dots,A_d\})$, determine when a relation matrix
$A_i$ factors as a matrix product $A_i=A_jA_k$ of other adjacency matrices from the same scheme.
\end{op}
We refer to such decompositions as \emph{factorizations within an association scheme}.
They generalize our group-theoretic Cayley-by-Cayley case, and they connect combinatorial
factorizations with algebraic structure inside the Bose–Mesner algebra of~$\mathfrak X$.
\item Let $G$ be a simple graph of order $n$.
We say that $A(G)$ admits an \defin{$\varepsilon$-factorization} if there exist simple graphs $H$ and $K$ on the same vertex set $V(G)$ with adjacency matrices $A(H),A(K)$ such that there exists a $0$--$1$ \emph{error matrix} $E\in\{0,1\}^{n\times n}$ with
\[
\|E\|_0\le \varepsilon n^2
\quad\text{and}\quad
\bigl(A(H)A(K)-A(G)\bigr)\odot (J-E)=0,
\]
where $\|E\|_0\coloneqq |\{(i,j):E_{ij}=1\}|$, $J$ is the all-ones matrix, and $\odot$ denotes the Hadamard (entrywise) product.

\begin{op}\label{prob:approx-factorization}
Given a simple graph $G$ and a tolerance $\varepsilon\in[0,1]$, decide whether $A(G)$ admits an $\varepsilon$-factorization that is, determine whether there exist simple graphs $H,K$ on $V(G)$ with
\[
|\{(i,j)\mid \ (A(H)A(K))_{ij}\neq A(G)_{ij}\}|\ \le\ \varepsilon n^2.
\]
\end{op}
\item Let $\mathcal H=(V,E)$ be a simple $k$-uniform hypergraph on the vertex set $V$, and let
$B(\mathcal H)\in\{0,1\}^{V\times E}$ be its incidence matrix, defined by
\[
B(\mathcal H)_{v,e} \;=\;
\begin{cases}
1,& \text{if } v\in e,\\
0,& \text{if } v\notin e.
\end{cases}
\]
Define the (pair–coincidence) adjacency matrix of $\mathcal H$ by
\[
A(\mathcal H)\;\coloneqq \;B(\mathcal H)\,B(\mathcal H)^{\!\top}\;-\;\operatorname{diag}\!\bigl(B(\mathcal H)\,B(\mathcal H)^{\!\top}\bigr).
\]
Thus, for $u\neq v$,
$A(\mathcal H)_{uv}\;=\;\bigl|\{\,e\in E\mid\ \{u,v\}\subseteq e\,\}\bigr|$ and $A(\mathcal H)_{uu}=0$.

We are interested in the factorization problem for $k$-uniform hypergraphs under this adjacency:
in particular, in the case of block designs. If $\mathcal H$ is a $2$-$(v,k,\lambda)$ design
(i.e., every pair of distinct vertices lies in exactly $\lambda$ blocks of size $k$), then
\[
A(\mathcal H)=\lambda\,(J-I),
\]
where $J$ and $I$ denote the $v\times v$ all-ones and identity matrices, respectively.
It is worth mentioning that this case has been studied here \citet{li2025cyclotomic,kreher2025lambda}.

\begin{op}Classify the factorizations of $\lambda(J-I)$ as a product of adjacency
matrices of (simple, uniform) hypergraphs on $V$; that is, determine all pairs
$\mathcal H_1,\mathcal H_2$ on $V$ such that
\[
A(\mathcal H_1)\,A(\mathcal H_2)\;=\;\lambda\,(J-I).
\]
\end{op}
\item
Let $G$ be a finite group with $|G|=n$ and identity element $e$. For subsets $S,T\subseteq G$ and $u\in G$, recall that 
$N_{S,T}(u)\;\coloneqq \;\bigl|\{(s,t)\in S\times T\mid \ st=u\}\bigr|$.
Given integers $a,b\ge 1$, consider subsets $S,T\subseteq G$ with $|S|=a$ and $|T|=b$.

\begin{op}
Determine the largest absolute constant $c>0$ (independent of $G$ and $n$) such that for every finite group $G$ of order $n$ and all $a,b\ge 1$ with
$ab\ \le\ c\,\sqrt{n}$,
there exist subsets $S,T\subseteq G$ with $|S|=a$, $|T|=b$, and
\[
N_{S,T}(u)\ \le\ 1\qquad\text{for all }u\in G.
\]
\end{op}
\begin{op}
Find the asymptotically sharp threshold on the product $ab$ guaranteeing that one can choose $S,T\subseteq G\sm\{1\}$ so that $U\coloneqq ST$ has the unique–product property. In particular, when $e\notin S\cup T$ and $S,T$ are symmetric, this is equivalent to the Cayley matrix-product factorization
$A(G;U)\;=\;A(G;S)\,A(G;T)$.
\end{op}

\item Motivated by \citet{miraftab2025factorability}'s study of matrix‑product factorizations in the infinite setting, we ask to investigate the case of Cayley graphs of infinite groups.
\end{enumerate}

\section*{Acknowledgment}

This research was conducted while the second author was visiting the University of Regina. 
This research was supported by NSERC.

\section*{Data Availability} No datasets were generated or analysed during the current study.
\section*{Declarations}
The authors declare no conflict of interest.

\bibliographystyle{plainurlnat}
\bibliography{MPF.bib}

\begin{thebibliography}{11}
\providecommand{\natexlab}[1]{#1}
\providecommand{\url}[1]{\texttt{#1}}
\providecommand{\urlprefix}{URL }
\expandafter\ifx\csname urlstyle\endcsname\relax
  \providecommand{\doi}[1]{\href{https://dx.doi.org/#1}{\nolinkurl{doi:#1}}}\else
  \providecommand{\doi}[1]{\href{https://dx.doi.org/#1}{\nolinkurl{doi:#1}}}\fi
\providecommand{\eprint}[2][]{\href{https://dx.doi.org/10.48550/arXiv.#2}{\nolinkurl{doi:10.48550/arXiv.#2}}}

\bibitem[{Akbari et~al.(2025{\natexlab{a}})Akbari, Elahimanesh, and
  Miraftab}]{prime}
Saieed Akbari, Parsa Elahimanesh, and Babak Miraftab.
\newblock Prime factorization of graphs.
\newblock \emph{preprint}, 2025{\natexlab{a}}.

\bibitem[{Akbari et~al.(2025{\natexlab{b}})Akbari, Fan, Hu, Miraftab, and
  Wang}]{akbari2025spectral}
Saieed Akbari, Yi-Zheng Fan, Fu-Tao Hu, Babak Miraftab, and Yi~Wang.
\newblock Spectral methods for matrix product factorization.
\newblock \emph{Linear Algebra Appl.}, 709:111--123, 2025{\natexlab{b}}.
\newblock \doi{10.1016/j.laa.2025.01.005}.

\bibitem[{de~Caen et~al.(1990)de~Caen, Gregory, Hughes, and Kreher}]{decaen}
D.~de~Caen, D.~A. Gregory, I.~G. Hughes, and D.~L. Kreher.
\newblock Near-factors of finite groups.
\newblock \emph{Ars Combin.}, 29:53--63, 1990.

\bibitem[{Curtis and Reiner(2006)}]{MR2215618}
Charles~W. Curtis and Irving Reiner.
\newblock \emph{Representation theory of finite groups and associative
  algebras}.
\newblock AMS Chelsea Publishing, Providence, RI, 2006.
\newblock \doi{10.1090/chel/356}.
\newblock Reprint of the 1962 original.

\bibitem[{Hammack et~al.(2011)Hammack, Imrich, and Klav\v{z}ar}]{MR2817074}
Richard Hammack, Wilfried Imrich, and Sandi Klav\v{z}ar.
\newblock \emph{Handbook of product graphs}.
\newblock Discrete Mathematics and its Applications (Boca Raton). CRC Press,
  Boca Raton, FL, second edition, 2011.
\newblock \doi{10.1201/b10959}.
\newblock With a foreword by Peter Winkler.

\bibitem[{Kreher et~al.(2025{\natexlab{a}})Kreher, Li, and
  Stinson}]{kreher2025lambda}
Donald~L Kreher, Shuxing Li, and Douglas~R Stinson.
\newblock $\lambda $-fold near-factorizations of groups.
\newblock \emph{arXiv preprint arXiv:2503.09325}, 2025{\natexlab{a}}.

\bibitem[{Kreher et~al.(2025{\natexlab{b}})Kreher, Paterson, and
  Stinson}]{KreherPatersonStinson2025}
Donald~L Kreher, Maura~B Paterson, and Douglas~R Stinson.
\newblock Near-factorizations of dihedral groups.
\newblock \emph{Designs, Codes and Cryptography}, pages 1--24,
  2025{\natexlab{b}}.
\newblock \doi{10.1007/s10623-025-01715-8}.

\bibitem[{Li and Momihara(2025)}]{li2025cyclotomic}
Shuxing Li and Koji Momihara.
\newblock Cyclotomic construction of $\lambda $-fold near-factorizations of
  cyclic groups.
\newblock \emph{arXiv preprint arXiv:2507.18045}, 2025.

\bibitem[{Maghsoudi et~al.(2025)Maghsoudi, Miraftab, and
  Suda}]{maghsoudi2023matrix}
Farzad Maghsoudi, Babak Miraftab, and Sho Suda.
\newblock On matrix product factorization of graphs.
\newblock \emph{J. Algebraic Combin.}, 61(1):Paper No. 12, 23, 2025.
\newblock \doi{10.1007/s10801-024-01377-0}.

\bibitem[{Manjunatha~Prasad et~al.(2013)Manjunatha~Prasad, Sudhakara, Sujatha,
  and Vinay}]{Manjunatha}
K.~Manjunatha~Prasad, G.~Sudhakara, H.~S. Sujatha, and M.~Vinay.
\newblock Matrix product of graphs.
\newblock In \emph{Combinatorial matrix theory and generalized inverses of
  matrices}, pages 41--55. Springer, New Delhi, 2013.
\newblock \doi{10.1007/978-81-322-1053-5\_4}.

\bibitem[{Miraftab et~al.(2026)Miraftab, Radjavi, and
  Suda}]{miraftab2025factorability}
Babak Miraftab, Heydar Radjavi, and Sho Suda.
\newblock On the factorability of infinite graphs.
\newblock \emph{Linear Algebra Appl.}, 728:409--418, 2026.
\newblock \doi{10.1016/j.laa.2025.09.011}.

\end{thebibliography}
\nocite{*}
\newpage
    
\end{document}